\documentclass[final,3p,times]{elsarticle}

\usepackage{amsmath,amssymb,amsthm}
\usepackage{hyperref}
\usepackage{microtype}
\usepackage{cleveref}
\usepackage[T1]{fontenc}
\usepackage[utf8]{inputenc}
\newtheorem{theorem}{Theorem}

\begin{document}
\makeatletter
\def\ps@pprintTitle{%
  \let\@oddhead\@empty
  \let\@evenhead\@empty
  \let\@oddfoot\@empty
  \let\@evenfoot\@empty
}
\makeatother
\begin{frontmatter}

\title{Lifespan Lower Estimates for a Strongly Damped Semilinear Wave Equation}

\author[FST]{Firas Kaabi}
\address[FST]{Faculty of Sciences of Tunis, University of Tunis El Manar}

\begin{abstract}
We consider a strongly damped semilinear wave equation with initial data prescribed as $(\varrho\phi,\varrho h)$, where the profiles are fixed and only the amplitude $\varrho>0$ is allowed to vary. The question addressed here is how this rescaling affects a guaranteed lower bound for the maximal existence time. We show that the solution exists at least on a time interval of length comparable to $\varrho^{-(p-2)}$. The proof is based on the growth of a quadratic phase-space norm: after the source term is estimated by the relevant Sobolev embedding, the problem reduces to a scalar differential inequality. The constants produced in the argument are independent of $\varrho$, so the dependence on the initial amplitude remains explicit throughout.
\end{abstract}

\end{frontmatter}

\section{Introduction}

The dynamics of damped wave equations are shaped by two competing effects. The damping term is meant to dissipate the motion, while the nonlinear source can reinforce the solution and, for large enough growth, lead to finite-time growth beyond the continuation regime. This opposition has been studied extensively in the literature on global existence, blow-up and decay, using tools such as energy inequalities, potential-well decompositions, concavity arguments and stabilization estimates; see \cite{Sattinger1968,Levine1974,Matsumura1976,Nakao1981,Haraux1985,Messaoudi2001,Messaoudi2003,Alabau2005,Alabau2010,Gazzola2006,Gerbi2011,LevineParkSerrin1998}.

In the present work, the profiles themselves are not changed. We prescribe the initial data by multiplying a fixed pair $(\phi,h)$ by the same amplitude $\varrho>0$, and we ask how this single parameter affects the time for which the solution is guaranteed to exist. The result is therefore not intended to provide a new blow-up condition. Its role is different: it gives an explicit lower bound, written directly in terms of $\varrho$, before the loss of boundedness described in \eqref{eq:blowup-assumption} can occur.

Let $D\subset\mathbb{R}^n$ be a bounded domain with smooth boundary. The equation considered in this paper is
\begin{equation}
\begin{cases}
\psi_{tt}-\Delta \psi-a\Delta \psi_t+b\psi_t=|\psi|^{p-2}\psi,
& \text{in } D\times(0,T),\\
\psi=0,
& \text{on } \partial D\times(0,T),\\
\psi(\cdot,0)=\varrho\phi,\qquad \psi_t(\cdot,0)=\varrho h,
& \text{in } D.
\end{cases}
\label{eq:problem}
\end{equation}
The parameter $\varrho$ is positive and measures the common amplitude of the initial displacement and velocity. The fixed profiles are assumed to satisfy
\begin{equation}
\phi\in H_0^1(D),\qquad h\in L^2(D),
\qquad \varrho>0.
\label{eq:initial}
\end{equation}
The coefficients in the damping part are taken so that
\begin{equation}
a\ge0,\qquad b>-a\lambda_1,
\label{eq:coeff}
\end{equation}
where $\lambda_1$ is the first eigenvalue of the Dirichlet Laplacian on $D$.

The exponent of the source is chosen so that the Sobolev embeddings used below are available. More precisely, we impose
\begin{equation}
2<p<\infty \qquad \text{if } n=1,2,
\label{eq:p1}
\end{equation}
and, for dimensions $n\ge3$,
\begin{equation}
2<p\le \frac{2n-2}{n-2}.
\label{eq:p2}
\end{equation}

Lower estimates for blow-up times often come from comparison procedures. One introduces an auxiliary functional, derives a differential inequality, and then obtains a lower bound that may remain written in integral form. This is the case, for instance, in the result of Sun, Guo and Gao \cite{SunGuoGao2014}, where the lower estimate is expressed through an integral quantity depending on the initial energy and the parameters of the comparison argument. Such a bound is useful, but it does not immediately provide a simple security interval when the whole initial state is rescaled by a single amplitude. The point of the present paper is to make this dependence explicit: for data $(\varrho\phi,\varrho h)$, the guaranteed time is written directly as a power of $\varrho$. Related lifespan questions for critical or scale-invariant wave equations can be found in \cite{TakamuraWakasa2014,Takamura2015,Wakasugi2014,LaiTakamuraWakasa2017}.

The proof is arranged to keep the amplitude visible throughout the estimate, in contrast with more elaborate quantitative bounds for strongly damped wave equations such as \cite{Bchatnia2026}. We measure the solution by a quadratic phase-space norm and use the admissible Sobolev embedding only to control the source term. Once this is done, the whole argument is reduced to a scalar inequality whose initial value contains the factor $\varrho^2$. After integration, this is exactly what produces the power $\varrho^{-(p-2)}$. No potential-well decomposition or auxiliary comparison functional is needed.

The lower bound is stated and proved in Section~\ref{sec:main}. Section~\ref{sec:conclusion} closes the paper with a short summary.

\section{Main result}\label{sec:main}

Before the theorem is stated, let us point out the mechanism behind the estimate. Once the profiles $\phi$ and $h$ are fixed, the amplitude $\varrho$ enters the proof only through the initial size of the solution. The later estimates involve the domain, the damping coefficients and the Sobolev constant, but not a new power of $\varrho$. Thus the final dependence on the amplitude can be read directly from the initial value of the quadratic quantity used below.

\begin{theorem}\label{thm:main}
Let the assumptions \eqref{eq:initial}--\eqref{eq:p2} be satisfied, and assume that the prescribed profiles are nontrivial in the sense that
\begin{equation}
\|\nabla \phi\|_2^2+\|h\|_2^2>0.
\label{eq:nontrivial-data}
\end{equation}
Consider a solution $\psi$ of \eqref{eq:problem}, defined on its maximal interval $[0,T^*)$, such that
\begin{equation}
\psi\in C([0,T^*);H_0^1(D)),
\qquad
\psi_t\in C([0,T^*);L^2(D)).
\label{eq:solution-class}
\end{equation}
Assume moreover that any finite endpoint of this interval is detected by the loss of the $L^p(D)$ bound, namely
\begin{equation}
\lim_{t\uparrow T^*}\|\psi(t)\|_{L^p(D)}=\infty.
\label{eq:blowup-assumption}
\end{equation}
Then the maximal time satisfies the lower estimate
\begin{equation}
T^*\ge C\varrho^{-(p-2)},
\label{eq:main-lower-bound}
\end{equation}
where $C>0$ may depend on $D$, $p$, $a$, $b$, $\phi$ and $h$, but not on the amplitude $\varrho$.
\end{theorem}

\begin{proof}
We begin with the phase-space quantity
\begin{equation}
\mathcal{Y}(t)
=
\|\nabla \psi(t)\|_2^2+\|\psi_t(t)\|_2^2,
\qquad 0\le t<T^*.
\label{eq:Y-def}
\end{equation}
The restrictions on $p$ ensure the continuous embedding
$H_0^1(D)\hookrightarrow L^p(D)$. Hence,
\begin{equation}
\|\psi(t)\|_p
\le C\|\nabla \psi(t)\|_2
\le C\mathcal{Y}(t)^{1/2},
\qquad 0\le t<T^*.
\label{eq:Lp-control}
\end{equation}
Thus, under the continuation-loss assumption \eqref{eq:blowup-assumption}, the quantity $\mathcal{Y}$ cannot remain bounded near $T^*$; more precisely,
\begin{equation}
\mathcal{Y}(t)\to\infty
\qquad \text{as } t\uparrow T^*.
\label{eq:Y-blowup}
\end{equation}

We now control the possible increase of $\mathcal{Y}$. At the level of smooth approximations, we multiply the equation in \eqref{eq:problem} by $2\psi_t$ and integrate over $D$. This gives
\begin{equation}
\mathcal{Y}'(t)
=
2\int_D \psi_t(\psi_{tt}-\Delta\psi)\,dx.
\label{eq:Yprime-basic}
\end{equation}
For a solution with the regularity stated in \eqref{eq:solution-class}, the identity is justified by a density argument. More precisely, one first works with Galerkin approximations, for which the calculation is legitimate term by term. The resulting identity is then passed to the limit in the energy topology
$$
C([0,T];H_0^1(D))\times C([0,T];L^2(D)).
$$
This standard approximation procedure allows the energy relation to be used without assuming extra smoothness of the solution; see \cite{Lions1969}.

Using the equation itself, the term inside the integral in \eqref{eq:Yprime-basic} can be written as
\begin{equation}
\psi_{tt}-\Delta\psi
=
a\Delta\psi_t-b\psi_t+|\psi|^{p-2}\psi .
\label{eq:equation-rewritten}
\end{equation}
We therefore obtain
\begin{equation}
\mathcal{Y}'(t)
=
-2a\|\nabla\psi_t(t)\|_2^2
-2b\|\psi_t(t)\|_2^2
+
2\int_D |\psi|^{p-2}\psi\,\psi_t\,dx .
\label{eq:Yprime-full}
\end{equation}
The two linear terms do not create growth. In fact, the assumption \eqref{eq:coeff}, combined with Poincar\'e's inequality, gives
\begin{equation}
a\|\nabla\psi_t(t)\|_2^2+b\|\psi_t(t)\|_2^2
\ge
(a\lambda_1+b)\|\psi_t(t)\|_2^2
\ge0.
\label{eq:damping-positive}
\end{equation}
Consequently, after dropping this nonnegative dissipative contribution, the derivative of $\mathcal{Y}$ is bounded only by the source term:
\begin{equation}
\mathcal{Y}'(t)
\le
2\int_D |\psi|^{p-1}|\psi_t|\,dx .
\label{eq:Yprime-source}
\end{equation}
We now estimate the source term. H\"older's inequality gives
\begin{equation}
\int_D |\psi|^{p-1}|\psi_t|\,dx
\le
\|\psi(t)\|_{2p-2}^{p-1}\|\psi_t(t)\|_2.
\label{eq:holder}
\end{equation}
The exponent $2p-2$ is the reason for the restriction imposed on $p$. Indeed, when $n\ge3$, condition \eqref{eq:p2} is equivalent to the admissibility of the embedding
\begin{equation}
H_0^1(D)\hookrightarrow L^{2p-2}(D).
\label{eq:sobolev-2p}
\end{equation}
In dimensions one and two, this embedding holds for every finite Lebesgue exponent. Hence, in all admissible cases,
\begin{equation}
\|\psi(t)\|_{2p-2}
\le C\|\nabla\psi(t)\|_2.
\label{eq:L2p-control}
\end{equation}
Combining \eqref{eq:Yprime-source}, \eqref{eq:holder} and \eqref{eq:L2p-control}, we obtain
\begin{equation}
\mathcal{Y}'(t)
\le
C\|\nabla\psi(t)\|_2^{p-1}\|\psi_t(t)\|_2.
\label{eq:Yprime-gradient}
\end{equation}
Both factors on the right-hand side are controlled by $\mathcal{Y}(t)^{1/2}$, namely
\begin{equation}
\|\nabla\psi(t)\|_2\le \mathcal{Y}(t)^{1/2},
\qquad
\|\psi_t(t)\|_2\le \mathcal{Y}(t)^{1/2}.
\label{eq:Y-components}
\end{equation}
Thus the previous estimate reduces to the scalar inequality
\begin{equation}
\mathcal{Y}'(t)
\le
C\mathcal{Y}(t)^{p/2},
\qquad 0<t<T^*.
\label{eq:Y-differential}
\end{equation}
The constant $C$ may change from line to line, but it depends only on the fixed quantities in the problem and never on the amplitude $\varrho$.

We integrate \eqref{eq:Y-differential} by applying it to the decreasing power
\begin{equation}
\mathcal{Z}(t)=\mathcal{Y}(t)^{-\frac{p-2}{2}}.
\label{eq:Z-def}
\end{equation}
For almost every $t<T^*$, differentiation gives
\begin{equation}
\mathcal{Z}'(t)
=
-\frac{p-2}{2}\mathcal{Y}(t)^{-\frac p2}\mathcal{Y}'(t)
\ge -C.
\label{eq:Zprime}
\end{equation}
Therefore,
\begin{equation}
\mathcal{Z}(t)\ge \mathcal{Z}(0)-Ct,
\qquad 0<t<T^*.
\label{eq:Z-integrated}
\end{equation}
Since \eqref{eq:Y-blowup} implies $\mathcal{Z}(t)\to0$ as $t\uparrow T^*$, passing to the endpoint in \eqref{eq:Z-integrated} yields
\begin{equation}
T^*\ge C\mathcal{Z}(0).
\label{eq:T-lower-Z}
\end{equation}

It only remains to write this initial quantity in terms of $\varrho$. From \eqref{eq:problem},
\begin{equation}
\mathcal{Y}(0)
=
\|\nabla(\varrho\phi)\|_2^2+\|\varrho h\|_2^2
=
\varrho^2\bigl(\|\nabla\phi\|_2^2+\|h\|_2^2\bigr).
\label{eq:Y0}
\end{equation}
Consequently,
\begin{equation}
\mathcal{Z}(0)
=
\varrho^{-(p-2)}
\bigl(\|\nabla\phi\|_2^2+\|h\|_2^2\bigr)^{-\frac{p-2}{2}}.
\label{eq:Z0}
\end{equation}
The profile-dependent factor is finite and strictly positive by \eqref{eq:nontrivial-data}. Combining \eqref{eq:T-lower-Z} with \eqref{eq:Z0} proves \eqref{eq:main-lower-bound}.
\end{proof}

\section{Conclusion}\label{sec:conclusion}

We have proved a lower bound that displays explicitly the effect of the amplitude in the initial data. For the strongly damped semilinear wave equation studied here, the choice of initial state $(\varrho\phi,\varrho h)$ gives a maximal existence time bounded from below by a constant multiple of $\varrho^{-(p-2)}$. The constant may depend on the domain, the coefficients, the exponent and the fixed profiles, but it is independent of the scaling parameter.

The argument relies on a single phase-space quantity. The damping assumption gives the linear part the required sign, while the source term is controlled by the Sobolev embedding associated with the admissible exponent range. The rest of the proof is reduced to the integration of a scalar differential inequality. In this form, the result gives a direct guaranteed existence interval in terms of the amplitude and avoids the implicit integral form that often appears in lower blow-up-time estimates.

\end{document}